\documentclass[11pt]{article}
\usepackage{amssymb,amsthm}
\usepackage{amsmath,latexsym,epsf,mathabx,underscore}
\usepackage[sort&compress,numbers]{natbib} 
\usepackage{changepage}
\usepackage[all]{xy}
\usepackage{hyperref}
\numberwithin{equation}{section}
\usepackage{color}
\begin{document}

\newcommand{\nc}{\newcommand}
\def\PP#1#2#3{{\mathrm{Pres}}^{#1}_{#2}{#3}\setcounter{equation}{0}}
\def\mr#1{{{\mathrm{#1}}}\setcounter{equation}{0}}
\def\mc#1{{{\mathcal{#1}}}\setcounter{equation}{0}}
\def\mb#1{{{\mathbb{#1}}}\setcounter{equation}{0}}
\def\Mcc{\mc{C}}
\def\Mbe{\mb{E}}
\def\Mcp{\mc{P}}
\def\Mcg{\mc{G}}
\def\fbzh{\mc{C}(-,\mc{P}(\xi))\text{-exact}}
\def\extri{(\mc{C},\mb{E},\mathfrak{s})}
\def\Gproj{\xi\text{-}\mc{G}\text{projective}}
\def\Ginj{\xi\text{-}\mc{G}\text{injective}}
\def\GP{\mc{G}\mc{P}(\xi)}
\def\GI{\mc{G}\mc{I}(\xi)}
\def\P{\mc{P}(\xi)}
\def\Gpd{\xi\text{-}\mc{G}\text{pd}}
\def\Extri{\mb{E}\text{-triangle}}
\def\SGP{\mc{SGP}(\xi)}
\def\nSGP{\text{$n$}\text{-}\mc{SGP}(\xi)}
\def\mSGP{\text{$m$}\text{-}\mc{SGP}(\xi)}
\def\ext{\xi \text{xt}_{\xi}}
\def\extgp{\xi \text{xt}_{\mc{GP}(\xi)}}
\def\extgi{\xi \text{xt}_{\mc{GI}(\xi)}}
\def\gext{\mc{G}\xi\text{xt}_{\xi}}
\def\nGpd{n\text{-}\xi\text{-}\mc{SG}\text{pd}}
\def\ext{\xi \text{xt}_{\xi}}
\newtheorem{defn}{\bf Definition}[section]
\newtheorem{cor}[defn]{\bf Corollary}
\newtheorem{prop}[defn]{\bf Proposition}
\newtheorem{thm}[defn]{\bf Theorem}
\newtheorem{lem}[defn]{\bf Lemma}
\newtheorem{rem}[defn]{\bf Remark}
\newtheorem{exam}[defn]{\bf Example}
\newtheorem{fact}[defn]{\bf Fact}
\newtheorem{cond}[defn]{\bf Condition}
\def\Pf#1{{\noindent\bf Proof}.\setcounter{equation}{0}}
\def\>#1{{ $\Rightarrow$ }\setcounter{equation}{0}}
\def\<>#1{{ $\Leftrightarrow$ }\setcounter{equation}{0}}
\def\bskip#1{{ \vskip 20pt }\setcounter{equation}{0}}
\def\sskip#1{{ \vskip 5pt }\setcounter{equation}{0}}
\def\mskip#1{{ \vskip 10pt }\setcounter{equation}{0}}
\def\bg#1{\begin{#1}\setcounter{equation}{0}}
\def\ed#1{\end{#1}\setcounter{equation}{0}}




\title{\bf Gorenstein Derived Functors for Extriangulated Categories}

\smallskip
\author{\small Zhenggang He\\
\small E-mail:~zhenggang_he@163.com\\
\small Institute of Mathematics, School of Mathematics Sciences\\
\small Nanjing Normal University, Nanjing \rm210023 China}
\date{}
\maketitle
\baselineskip 15pt
%
%
%
\vskip 10pt%
\noindent {\bf Abstract}:
Let $(\Mcc,\Mbe,\mathfrak{s})$ be an extriangulated category with a proper class $\xi$ of $\Mbe$-triangles.
In this paper, we study Gorenstein  derived functors  for extriangulated categories. More precisely, we first
introduce the notion of the proper $\Gproj$ resolution  for any object in $\mathcal{C}$ and define the functors $\extgp$ and $\extgi$. Under some assumptions, we 
 give  some  equivalent  characterizations  for  any  object with
   finite $\Gproj$ dimension. Next we get some nice results by 
   using derived functors. As an application, our main results generalize their work 
   by  Ren-Liu. Moreover, our proof is not far from the usual module categories or triangulated
 categories.

\mskip\





%
%
\vskip 30pt

\section{Introduction}
\quad~The notion of extriangulated categories was introduced by Nakaoka and Palu in \cite{HY} as a
simultaneous generalization of exact categories and triangulated categories.  Hence many results hold on exact categories and triangulated
categories can be unified in the same framework.
Let $(\Mcc,\Mbe, \mathfrak{s})$ be an extriangulated category.  Hu, Zhang and Zhou 
\cite{JDP} studied a relative homological
algebra in $\Mcc$ which parallels the relative homological algebra in a triangulated category. By
specifying a class of $\Mbe$-triangles, which is called a proper class $\xi$ of $\Mbe$-triangles, they introduced
$\xi$-$\mathcal{G}$projective dimensions and $\xi$-$\mathcal{G}$injective dimensions and discussed their properties.  He  \cite{He} introduced the
notion of the$\xi$-$\mathcal{G}$projectiHeve resolution, and study the relation between $\xi$-projective resolution
and $\xi$-$\mathcal{G}$projective resolution for any object $A\in\mathcal{C}$.

In the category of modules, Holm \cite{H1} introduced the Gorenstein derived 
functors ${\text{Ext}_{\mathcal{GP}}^n}(-,-)$ and  ${\text{Ext}_{\mathcal{GI}}^n}(-,-)$ 
of $\text{Hom}_R(-,-)$, and proved that the functorial isomorphisms  
$${\text{Ext}_{\mathcal{GP}}^n}(M,N)\backsimeq {\text{Ext}_{\GI}^n}(M,N) $$
 hold over arbitrary rings $R$, provided that $\mathcal{G}\text{pd}_R M<\infty $ and
 $\mathcal{G}\text{id}_R N<\infty$. This result generalized that of Enochs and Jenda.

Inspired by Holm, Ren and Liu \cite{WZ} defined $\xi$-Gorenstein derived functors $\gext^n(-,-)$ 
of $\text{Hom}_{\Mcc}(-,-)$ in the triangulated categories, and they further studied Gorenstein
 homological dimensions for triangulated categories with $\xi$-Gorenstein derived functors,
 and proved that the functorial isomorphisms  
 $${\xi\text{xt}_{\GP}^n}(M,N)\backsimeq {\xi\text{xt}_{\GI}^n}(M,N) $$
 with $\xi$-$\mathcal{G}\text{pd} M<\infty $ and
$\xi$-$\mathcal{G}\text{id} N<\infty$.

What is more, Hu, Zhang and Zhou\cite{JZP} gave some characterizations of $\xi$-$\mathcal{G}$projective dimension by using derived functors in $\Mcc$. Let $\P$ (resp. $\mathcal{I}(\xi)$) be a generating (resp.cogenerating) subcategory of $\mathcal{C}$. They proved that the following equality holds under
some assumptions:
$$\text{sup}\left\{\Gpd M\  |\  \text{for any}\  M \in \mathcal{C}\right\} = \text{sup}\left\{\xi\text{-}\mathcal{G}\text{id} \ M \ |\  \text{for any} M \in \mathcal{C}\right\},$$
where $\Gpd M$ (resp. $\xi\text{-}\mathcal{G}\text{id} M $) denotes $\xi$-$\mathcal{G}$projective (resp. $\xi$-$\mathcal{G}$injective) dimensionn of $M$. And they also pointed  that 

$(1)$ if $M$ is an object in $\mathcal{I}(\xi)$, then $\Gpd M = \xi$-pd$M$;

$(2)$ if $M$ is an object in $\P$, then $\xi$-$\mathcal{G}$id $M = \xi$-id$M$.

This paper is devoted to further study Gorenstein homological properties for extriangulated categories proposed by \cite{JZP}. More precisely, we introduce the notion of the proper $\Gproj$ resolution  for any object in $\mathcal{C}$, then show the Horseshoe Lemma
and Comparison Theorem of the proper $\Gproj$ resolution.
Inspired by \cite{H1} and  \cite{WZ}, we define the functors 
$\extgp$ and $\extgi$, then there exist the long exact sequences 
of them. If $A\in\mathcal{GP}^*(\xi)$, $B\in \mathcal{GI}^*(\xi)$,
 we have proved that $$\extgp^n(A,B)=\extgi^n(A,B).$$ In this situation we get the $\xi$-Gorenstein derived functors $\gext^i(-,-)$ of $\Mcc(-,-)$, where the first argument is for objects in $\mathcal{GP}^*(\xi)$ and the second argument is for objects in $\mathcal{GI}^*(\xi)$). Moreover, We give some equivalent characterizations for any object in $\mathcal{GP}^*(\xi)$ with finite $\Gproj$ dimension and any object in $\mathcal{GI}^*(\xi)$ with finite $\Ginj$ dimension.

Note that module categories and triangulated categories can be viewed as extriangulated 
categories. As an application, our main results generalize their work by Ren-Liu.

\section{Preliminaries}

\quad~ In this section, we briefly recall some basic definitions of extriangulated categories from \cite{HY}. We omit some details here, but the reader can find them in \cite{HY}.

\begin{defn}\citep[Definition 2.1]{HY}
	Suppose that $\mathcal{C}$ is equipped with an additive bifunctor $\mathbb{E}:\mathcal{C}^{op}\times \mathcal{C}\rightarrow \mathbf{Ab}$. For any pair of objects $A, C$ in $\mathcal{C}$, an element $\delta \in\mathbb{E}(C,A)$ is called an $\mathbb{E}$-\emph{extension}. Thus formally, an $\mathbb{E}$-extension is a triplet $(A,\delta,C)$. Since $\mathbb{E}$ is a functor, for any $a\in \mathcal{C}(A, A')$ and $c\in \mathcal{C}(C, C)$, we have $\mathbb{E}$-extensions $\mathbb{E}(C, a)(\delta)\in \mathbb{E}(C, A') $ and $\mathbb{E}(c,A)(\delta)\in \mathbb{E}(C', A).$ We abbreviately denote them by $a_{*}\delta$ and $c^* \delta$ respectively. In this terminology, we have
	$$ \mathbb{E}(c, a)(\delta)=c^*a_*\delta=a_*c^*\delta$$
	in $\mathbb{E}(C', A')$. For any $ A$, $ C\in \mathcal{C}$, the zero
	element $0\in\mathbb{E}(C, A)$ is called the split $\mathbb{E}$-extension.
\end{defn}

\begin{defn}\citep[Definition 2.7]{HY}
	Let $A$, $C\in \mathcal{C}$ be any pair of objects. Two sequences of morphisms $A\stackrel{x}\longrightarrow B\stackrel{y}\longrightarrow C$ and $A\stackrel{x'}\longrightarrow B'\stackrel{y'}\longrightarrow C$ in $\mathcal{C}$ are said to be \emph{equivalent} if there exists an isomorphism $b\in \mathcal{C}(B,B')$ which makes the following diagram commutative.
	$$
	\xymatrix{
		A\ar@{=}[d]\ar[r]^x &B\ar[d]^b_{\simeq}\ar[r]^y &C\ar@{=}[d]\\
		A\ar[r]^{x'} &B'\ar[r]^{y'} &C}
	$$
\end{defn}	

We denote the equivalence class of $A\stackrel{x}\longrightarrow B\stackrel{y}\longrightarrow C$ by [$A\stackrel{x}\longrightarrow B\stackrel{y}\longrightarrow C$].

\begin{defn}\citep[Definition 2.8]{HY}
	(1) For any $A$, $C\in \mathcal{C}$, we denote as
	$$0=[A \stackrel{\begin{tiny}\begin{bmatrix}
		1 \\
		0
		\end{bmatrix}\end{tiny}}{\longrightarrow} A \oplus C \stackrel{\begin{tiny}\begin{bmatrix}
		0&1
		\end{bmatrix}\end{tiny}}{\longrightarrow} C].$$
	(2) For any $[A\stackrel{x}\longrightarrow B\stackrel{y}\longrightarrow C]$ and $[A'\stackrel{x'}\longrightarrow B'\stackrel{y'}\longrightarrow C']$, we denote as
	$$
	[A\stackrel{x}\longrightarrow B\stackrel{y}\longrightarrow C]\oplus[A'\stackrel{x'}\longrightarrow B'\stackrel{y'}\longrightarrow C']=[A\oplus A'\stackrel{x\oplus x'}\longrightarrow B\oplus B'\stackrel{y\oplus y'}\longrightarrow C\oplus C'].
	$$
\end{defn}

\begin{defn}\citep[Definition 2.9]{HY}
	Let $\mathfrak{s}$ be a correspondence which associates an equivalence class $\mathfrak{s}(\delta)=[A\stackrel{x}{\longrightarrow}B\stackrel{y}{\longrightarrow}C]$ to any $\mathbb{E}$-extension $\delta\in\mathbb{E}(C,A)$ . This $\mathfrak{s}$ is called a  \emph{realization} of $\mathbb{E}$, if for any morphism $(a,c):\delta\rightarrow\delta'$ with $\mathfrak{s}(\delta)=[A\stackrel{x}\longrightarrow B\stackrel{y}\longrightarrow C]$ and $\mathfrak{s}(\delta')=[A'\stackrel{x'}\longrightarrow B'\stackrel{y'}\longrightarrow C']$, there exists $b\in \mathcal{C}$ which makes the following  diagram commutative
	$$\xymatrix{
		& A \ar[d]_-{a} \ar[r]^-{x} & B  \ar[r]^{y}\ar[d]_-{b} & C \ar[d]_-{c}    \\
		&A'\ar[r]^-{x'} & B' \ar[r]^-{y'} & C'.    }
	$$
	In the above situation, we say that the triplet $(a,b,c)$ realizes $(a,b)$.
\end{defn}

\begin{defn}\citep[Definition 2.10]{HY}
	Let $\mathcal{C},\mathbb{E}$ be as above. A realization $\mathfrak{s}$ of $\mathbb{E}$ is said to be {\em additive} if it satisfies the following conditions.
	
	(a) For any $A,~C\in\mathcal{C}$, the split $\mathbb{E}$-extension $0\in\mathbb{E}(C,A)$ satisfies $\mathfrak{s}(0)=0$.
	
	(b) $\mathfrak{s}(\delta\oplus\delta')=\mathfrak{s}(\delta)\oplus\mathfrak{s}(\delta')$ for any pair of $\mathbb{E}$-extensions $\delta$ and $\delta'$.
	
\end{defn}

\begin{defn}\citep[Definition 2.12]{HY}
	A triplet $(\mathcal{C}, \mathbb{E},\mathfrak{s})$ is called an \emph{ extriangulated category} if it satisfies the following conditions. \\
	$\rm(ET1)$ $\mathbb{E}$: $\mathcal{C}^{op}\times\mathcal{C}\rightarrow \mathbf{Ab}$ is a biadditive functor.\\
	$\rm(ET2)$ $\mathfrak{s}$ is an additive realization of $\mathbb{E}$.\\
	$\rm(ET3)$ Let $\delta\in\mathbb{E}(C,A)$ and $\delta'\in\mathbb{E}(C',A')$ be any pair of $\mathbb{E}$-extensions, realized as
	$$\mathfrak{s}(\delta)=[A\stackrel{x}{\longrightarrow}B\stackrel{y}{\longrightarrow}C] \quad \text{and} \quad \mathfrak{s}(\delta')=[A'\stackrel{x'}{\longrightarrow}B'\stackrel{y'}{\longrightarrow}C'].$$ For any commutative square	
    $$\xymatrix{
		A \ar[d]_{a} \ar[r]^{x} & B \ar[d]_{b} \ar[r]^{y} & C \\
		A'\ar[r]^{x'} &B'\ar[r]^{y'} & C'}$$
	in $\mathcal{C}$, there exists a morphism $(a,c)$: $\delta\rightarrow\delta'$ which is realized by $(a,b,c)$.\\
	$\rm(ET3)^{op}$ Dual of $\rm(ET3)$.\\
    $\rm(ET4)$ Let $\delta\in \mathbb{E}(D,A)$ and $\delta'\in \mathbb{E}(F,B)$ be $\mathbb{E}$-extensions respectively realized by
	$$A\stackrel{f}{\longrightarrow}B\stackrel{f'}{\longrightarrow}D\quad \text{and }\quad B\stackrel{g}{\longrightarrow}C\stackrel{g'}{\longrightarrow}F.$$
	Then there exist an object $E\in\mathcal{C}$, a commutative diagram
	$$\xymatrix{
		A \ar@{=}[d]\ar[r]^{f} &B\ar[d]_{g} \ar[r]^{f'} & D\ar[d]^{d} \\
		A \ar[r]^{h} & C\ar[d]_{g'} \ar[r]^{h'} & E\ar[d]^{e} \\
		& F\ar@{=}[r] & F   }$$
	in $\mathcal{C}$, and an $\mathbb{E}$-extension $\delta''\in \mathbb{E}(E,A)$ realized by $A\stackrel{h}{\longrightarrow}C\stackrel{h'}{\longrightarrow}E$, which satisfy the following compatibilities.\\
	$(\textrm{i})$ $D\stackrel{d}{\longrightarrow}E\stackrel{e}{\longrightarrow}F$ realizes $f'_*\delta'$,\\
	$(\textrm{ii})$ $d^*\delta''=\delta$,\\
	$(\textrm{iii})$ $f_*\delta''=e^*\delta$.\\
	$\rm(ET4)^{op}$ Dual of $\rm(ET4)$.
\end{defn}
For examples of extriangulated categories, see \citep[Example 2.13]{HY} and \citep[Remark 3.3]{JDP}.

We will use the following terminology.

\begin{defn}\citep[Definition 2.15 and 2.19]{HY}
	Let $(\mathcal{C},\mathbb{E},\mathfrak{s})$ be an extriangulated category.
	
	(1) A sequence $A\stackrel{x}\longrightarrow B\stackrel{y}\longrightarrow C$ is called \emph{conflation} if it realizes some $\mathbb{E}$-extension $\delta\in \mathbb{E}(C,A)$. In this case, $x$ is called an\emph{ inflation} and $y$ is called a \emph{deflation}.
	
	(2) If a conflation $A\stackrel{x}\longrightarrow B\stackrel{y}\longrightarrow C$ realizes $\delta\in \mathbb{E}(C,A)$, we call the pair ($A\stackrel{x}\longrightarrow B\stackrel{y}\longrightarrow C, \delta$) an $\mathbb{E}$-\emph{triangle}, and write it by
	$$\xymatrix{
		A \ar[r]^{x} & B  \ar[r]^{y} & C  \ar@{-->}[r]^{\delta} & }.$$
	We usually don't write this "$\delta$" if it not used in the argument.
	
	(3) Let $\xymatrix{
		A \ar[r]^{x} & B  \ar[r]^{y} & C  \ar@{-->}[r]^{\delta} & }$ and $\xymatrix{
		A '\ar[r]^{x'} & B'  \ar[r]^{y'} & C'  \ar@{-->}[r]^{\delta'} & }$ be any pair of $\mathbb{E}$-triangles. If a triplet $(a,b,c)$ realizes $(a,c):\delta \rightarrow\delta'$, then we write it as
	
	$$\xymatrix{
		A \ar[d]_{a} \ar[r]^{x} & B  \ar[r]^{y}\ar[d]_{b} & C \ar[d]_{c} \ar@{-->}[r]^{\delta} &  \\
		A'\ar[r]^{x'} & B' \ar[r]^{y'} & C' \ar@{-->}[r]^{\delta'} &    }
	$$
	and call $(a,b,c)$ a\emph{ morphism }of $\mathbb{E}$-triangles.
	
	(4) An $\mathbb{E}$-triangle $\xymatrix{A\ar[r]^{x}&B\ar[r]^{y}&C\ar@{-->}[r]^{\delta}&}$ is called\emph{ split} if $\delta=0$.
\end{defn}

Following lemmas will be used many times in tthis paper.                              

\begin{lem}\label{FJ}\citep[Corollary 3.5]{HY}
	Assume that $(\mathcal{C},\mathbb{E},\mathfrak{s})$ satisfies $\rm(ET1)$, $\rm(ET2)$, $\rm(ET3)$ and $\rm(ET3)^{op}$. Let
	$$
	\xymatrix{
		A\ar[r]^x\ar[d]^a&B\ar[r]^y\ar[d]^b&C\ar[d]^c\ar@{-->}[r]^{\delta} & \\
		A'\ar[r]^{x'}&B'\ar[r]^{y'}&C'\ar@{-->}[r]^{\delta'} &
	}
	$$
	be any morphism of $\mathbb{E}$-triangles. Then the following are equivalent.
	
	(1) $a$ factors through $x$.
	
	(2) $a_*\delta=c^*\delta'=0$.
	
	(3) $c$ factors through $y'$.
\end{lem}	
	\noindent In particular, in the case $\delta = \delta'$ and $(a, b, c) =(1_A, 1_B, 1_C )$, we have
	$$
	x \text{\; is a section} \Leftrightarrow \delta \text{\; is split} \Leftrightarrow y \text{\;is a retraction}.
	$$

	\begin{lem}\label{BH}\citep[Proposition 3.15]{HY}
		Let $(\mathcal{C},\mathbb{E},\mathfrak{s})$ be an extriangulated category. Then the following hold.
		
		(1) Let $C$  be any object, and let $A_1\stackrel{x_1}{\longrightarrow}B_1\stackrel{y_1}{\longrightarrow}C\stackrel{\delta_1}\dashrightarrow$  and $A_2\stackrel{x_2}{\longrightarrow}B_2\stackrel{y_2}{\longrightarrow}C\stackrel{\delta_2}\dashrightarrow$  be any pair of $\mathbb{E}$-triangles. Then there is a commutative diagram in $\mathcal{C}$
		$$
		\xymatrix{
			&{A_2}\ar[d]^{m_2}\ar@{=}[r]&{A_2}\ar[d]^{x_2}\\
			{A_1}\ar@{=}[d]\ar[r]^{m_1}&M\ar[r]^{e_1}\ar[d]^{e_2}&{B_2}\ar[d]^{y_2}\\
			 {A_1}\ar[r]^{x_1}&{B_1}\ar[r]^{y_1}&C}
		$$
		\noindent which satisfies $\mathfrak{s}(y^*_2\delta_1)=[A_1\stackrel{m_1}{\longrightarrow}M\stackrel{e_1}{\longrightarrow}B_2]$ and $\mathfrak{s}(y^*_1\delta_2)=[A_2\stackrel{m_2}{\longrightarrow}M\stackrel{e_2}{\longrightarrow}B_1]$.
		
		(2)  Let $A$  be any object, and let $A\stackrel{x_1}{\longrightarrow}B_1\stackrel{y_1}{\longrightarrow}C_1\stackrel{\delta_1}\dashrightarrow$  and $A\stackrel{x_2}{\longrightarrow}B_2\stackrel{y_2}{\longrightarrow}C_2\stackrel{\delta_2}\dashrightarrow$  be any pair of $\mathbb{E}$-triangles. Then there is a commutative diagram in $\mathcal{C}$
		$$
		\xymatrix{
			A\ar[d]^{x_2}\ar[r]^{x_1}&{B_1}\ar[d]^{m_2}\ar[r]^{y_1}&{C_1}\ar@{=}[d]\\
			{B_2}\ar[d]^{y_2}\ar[r]^{m_1}&M\ar[r]^{e_1}\ar[d]^{e_2}&{C_1}\\
			{C_2}\ar@{=}[r]&{C_2}
		}
		$$
		\noindent which satisfies $\mathfrak{s}(x_{2_*}\delta_1)=[B_2\stackrel{m_1}{\longrightarrow}M\stackrel{e_1}{\longrightarrow}C_1]$ and $\mathfrak{s}(x_{1_*}\delta_2)=[B_1\stackrel{m_2}{\longrightarrow}M\stackrel{e_2}{\longrightarrow}C_2]$.
	\end{lem}

	Now we are in the position to introduce the concept for the
	 proper classes of $\mathbb{E}$-triangles following \cite{JDP}. 
	  We always assume that $(\mathcal{C},\mathbb{E},\mathfrak{s})$ is an extriangulated categroy.
\begin{defn}
	Let $\xi$ be a class of $\mathbb{E}$-triangles. One says $\xi$ is \emph{closed under base change} if for any $\mathbb{E}$-triangle
	$$
	\xymatrix{
		A\ar[r]^x&B\ar[r]^y&C\ar@{-->}[r]^{\delta}&
	}
	\in \xi
	$$
	\noindent and any morphism $c:C'\rightarrow C$, then any $\mathbb{E}$-triangle$\xymatrix{A\ar[r]^{x'}&B'\ar[r]^{y'}&C'\ar@{-->}[r]^{c^*\delta}&}$ belongs to $\xi$.
	
	Dually, one says $\xi$ is \emph{closed under cobase change} if for any $\mathbb{E}$-triangle
	$$
	\xymatrix{
		A\ar[r]^x&B\ar[r]^y&C\ar@{-->}[r]^{\delta}&
	}
	\in \xi
	$$
	\noindent and any morphism $a:A\rightarrow A'$, then any $\mathbb{E}$-triangle$\xymatrix{A'\ar[r]^{x'}&B'\ar[r]^{y'}&C\ar@{-->}[r]^{a_*\delta}&}$ belongs to $\xi$.
	
\end{defn}

\begin{defn}
	A class of $\mathbb{E}$-triangles $\xi$ is called \emph{saturated} if in the situation of Lemma \ref{BH}(1), when $\xymatrix{A_2\ar[r]^{x_2}&B_2\ar[r]^{y_2}&C\ar@{-->}[r]^{\delta_2}&}$ and $\xymatrix{A_1\ar[r]^{m_1}&M\ar[r]^{m_1}&B_2\ar@{-->}[r]^{y_2^*\delta_1}&}$ belong to $\xi$, then the
	$\mathbb{E}$-triangle $\xymatrix{A_1\ar[r]^{x_1}&B_1\ar[r]^{y_1}&C\ar@{-->}[r]^{\delta_1}&}$ belongs to  $\xi$.
\end{defn}

We denote the full subcategory consisting of the split $\mathbb{E}$-triangle by $\Delta_0$.
\begin{defn}\label{ZL}\citep[Definition 3.1]{JDP}
	Let $\xi$ be a class of $\mathbb{E}$-triangles which is closed under isomorphisms. $\xi$ is called a \emph{proper class } of $\mathbb{E}$-triangles if the following conditions holds:
	
	(1) $\xi$ is closed under finite coproducts and $\Delta_0 \subseteq \xi$.
	
	(2) $\xi$ is closed under base change and cobase change.
	
	(3) $\xi$ is saturated.
\end{defn}

\begin{defn}\citep[Definition 4.1]{JDP} An object $P\in \mathcal{C}$ is called $\xi$-\emph{projective} if for any $\mathbb{E}$-triangle
	
	$$
	\xymatrix{
		A\ar[r]^{x}&B\ar[r]^{y}&C\ar@{-->}[r]^{\delta}&
	}
	$$
	in $\xi$, the induced sequence of abelian groups
	$$
	0\longrightarrow\mathcal{C}(P,A)\longrightarrow \mathcal{C}(P,B)\longrightarrow\mathcal{C}(P,C)\longrightarrow 0
	$$
	is exact. Dually, we have the definition of $\xi$-injective.
\end{defn}

We denote $\P$ (resp. $\mathcal{I}(\xi)$) the class of $\xi$-projective 
(resp. $\xi$-injective) objects of $\Mcc$. An extriangulated category $(\mathcal{C},\mathbb{E},\mathfrak{s})$ is said 
	to have \emph{enough $\xi$-projectives } (resp. \emph{enough $\xi$-injectives }  ) 
	provided that for each object $A$ there exists an $
\mathbb{E}$-triangle $K\longrightarrow P\longrightarrow A\dashrightarrow$ 
(resp. $A\longrightarrow I\longrightarrow K\dashrightarrow$) in $\xi$ with
 $P\in\mathcal{P}(\xi)$(resp. $I\in\mathcal{I}(\xi)$).

 \begin{lem}\label{FBZH1}\citep[Lemma 4.2]{JDP}
	If $\mathcal{C}$ has enough $\xi$-projectives, then an $\mathbb{E}$-triangle $ A\longrightarrow B\longrightarrow C\dashrightarrow$ in $\xi$ if and only if induced sequence of abelian groups
	$$
	0\longrightarrow\mathcal{C}(P,A)\longrightarrow \mathcal{C}(P,B)\longrightarrow\mathcal{C}(P,C)\longrightarrow0
	$$
	is exact for all $P\in\mathcal{P}(\xi)$.
\end{lem}

The \emph{$\xi$-projective dimension} $\xi$-pd$A$ of an object $A$ is defined inductively. When $A=0$, put $\xi$-pd$A=-1$.  If $A\in\mathcal{P}(\xi)$, then define $\xi$-pd$A=0$. Next by induction, for an integer $n>0$, put $\xi$-pd$A\leqslant n$ if there exists an $\mathbb{E}$-triangle $K\rightarrow P\rightarrow A\dashrightarrow$ in $\xi$ with $P\in \mathcal{P}(\xi)$ and $\xi$-pd$K\leqslant n-1$.
We say $\xi$-pd$A=n$ if $\xi$-pd$A\leqslant n$ and $\xi$-pd$A \nleqslant  n-1$. 
If $\xi$-pd$A\neq n$, for all $n\geqslant0$, we set  $\xi$-pd$A=\infty$.
Dually we can define the $\xi$-injective dimension $\xi$-id$A$ of an object $A\in\Mcc$.

We use $\widehat{\mathcal{P}}(\xi)$ (resp. $\widehat{\mathcal{I}}(\xi)$) to 
denote the full subcategory of $\Mcc$ whose objects have finite $\xi$-projective 
 (resp. $\xi$-injective) dimension.

\begin{defn}\citep[Definition 4.4]{JDP}
	An complex $\mathbf{X}$ is called \emph{$\xi$-exact} if $\mathbf{X}$ is a diagram
	$$\xymatrix{
		\cdots\ar[r] &X_1\ar[r]^{d_1}&X_0 \ar[r]^{d_0}&X_{-1}\ar[r]&\cdots
	}
	$$
	in $\mathcal{C}$ such that for each integer $n$, there exists an $\mathbb{E}$-triangle $K_{n+1}\stackrel{g_n}\longrightarrow X_n\stackrel{f_n}\longrightarrow C\stackrel{\delta_n}\dashrightarrow$ in $\xi$ and $d_n=g_{n-1}f_n$. These $\mathbb{E}$-triangles are called the\emph{ resolution $\mathbb{E}$-triangles }of the $\xi$-exact complex $\mathbf{X}$.
\end{defn}

\begin{defn}\citep[Definition 4.5, 4.6]{JDP}
	Let $\mathcal{W}$ be a class of objects in $\Mcc$. An $\mathbb{E}$-triangle $A\longrightarrow B\longrightarrow C\dashrightarrow$ in $\xi$ is called to be $\Mcc(-,\mathcal{W})$\emph{-exact} (respectively $\Mcc(\mathcal{W},-)$\emph{-exact}) if for any $W\in\mathcal{W}$, the induced sequence of abelian group $0 \rightarrow \Mcc(C,W)\rightarrow\Mcc(B,W)\rightarrow\Mcc(A,W)\rightarrow0$ (respectively $0 \rightarrow \Mcc(W,A)\rightarrow\Mcc(W,B)\rightarrow\Mcc(W,C)\rightarrow0$) is exact in $\mathbf{Ab}$.

	A complex $\mathbf{X}$ is called $\Mcc(-,\mathcal{W})$-\emph{exact}  ( respectively $\Mcc(\mathcal{W},-)$-\emph{exact} ) if it is a $\xi$-exact complex with $\Mcc(-,\mathcal{W})$-exact resolution $\mathbb{E}$-triangles ( respectively $\Mcc(\mathcal{W},-)$-exact resolution $\mathbb{E}$-triangles ).
	
	A $\xi$-exact complex $\mathbf{X}$ is called \emph{complete $\Mcp(\xi)$-exact} if it is $\Mcc(-,\Mcp(\xi))$-exact.
\end{defn}
\begin{defn}
	An \emph{$\xi$-projective resolution} of an object $A\in\Mcc$ is a $\xi$-exact complex
	\[
	\xymatrix{
		\cdots\ar[r]&P_{n}\ar[r]&P_{n-1}\ar[r]&{\cdots}\ar[r]&P_{1}\ar[r]&P_{0}\ar[r]&A\ar[r]&0
	}
	\]
	in $\Mcc$ with $P_{n}\in\P$ for all $n\geqslant0$.
\end{defn}
\begin{defn}\citep[Definition 4.7, 4.8]{JDP}
	A \emph{complete $\xi$-projective resolution} is a complete $\Mcp(\xi)$-exact complex
	$$
	\mathbf{P}: \xymatrix{ \cdots\ar[r]&P_1\ar[r]^{d_1}&P_{0}\ar[r]^{d_0}&P_{-1}\ar[r]&{\cdots}}
	$$
	in $\Mcc$ such that $P_n$ is projective for each integer $n$ . 
	And for any $P_n$, there exists a $\Mcc(-,\Mcp(\xi))$-exact $\mathbb{E}$-triangle
	 $\xymatrix{ K_{n+1}\ar[r]^{g_n}&P_n\ar[r]^{f_n}&K_n\ar@{-->}[r]^{\delta_n}&}$ 
	 in $\xi$ which is the resolution $\mathbb{E}$-triangle of $\mathbf{P}$. 
	 Then the objects $K_n$ are called \emph{$\xi$-$\Mcg$projective} for each integer $n$. 
\end{defn}
Dually we can define the $\xi$-$\mathcal{G}$injective objects. We denote by
 $\Mcg\Mcp(\xi)$ (rsep. $\GI$) the subcategory of $\xi$-$\Mcg$projective ($\xi$-$\Mcg$injective)
  objects in $\Mcc$.

  \begin{defn}\citep[Definition 3.18]{He}\label{Gpr}
  A \emph{$\Gproj$ resolution} of an object $A\in\Mcc$ is a $\xi$-exact complex
  \[
  \xymatrix{
	  \cdots\ar[r]&G_{n}\ar[r]&G_{n-1}\ar[r]&{\cdots}\ar[r]&G_{1}\ar[r]&G_{0}\ar[r]&A\ar[r]&0
  }
  \]
  in $\Mcc$ such that $G_{n}\in\GP$ for all $n\geqslant0$. \emph{$\xi$-$\mathcal{G}$injective coresolution} of 
  an object $A\in\Mcc$ is defined by dual.
\end{defn}

The \emph{$\Gproj$ dimension} $\Gpd A$ of an object $A$ is defined inductively. When $A=0$, put $\Gpd A=-1$.  If $A\in\GP$, then define $\Gpd A=0$. Next by induction, for an integer $n>0$, put $\Gpd A\leqslant n$ if there exists an $\mathbb{E}$-triangle $K\rightarrow G\rightarrow A\dashrightarrow$ in $\xi$ with $G\in \GP$ and $\Gpd K\leqslant n-1$.
We say $\Gpd A=n$ if $\Gpd A\leqslant n$ and $\Gpd A \nleqslant  n-1$. If $\Gpd A\neq n$, for all $n\geqslant0$, we set  $\Gpd A=\infty$.
Dually we can define the  $\xi$-$\mathcal{G}$injective dimension $\xi$-$\mathcal{G}$id$A$ of an object $A\in\Mcc$.
  
We use $\widehat{\mathcal{GP}}(\xi)$ (resp. $\widehat{\mathcal{GI}}(\xi)$) to 
denote the full subcategory of $\Mcc$ whose objects have finite $\Gproj$ 
 (resp. $\xi$-$\mathcal{G}$injective) dimension.

Here we introduce the weak idempotent completeness condition for an extriangulated categories.
\begin{cond}[Condition (WIC)]\label{WIC}
	Consider the following conditions.
	
	(1) Let $f\in\Mcc(A,B),g\in\Mcc(B,C)$ be any composable  pair of morphisms. If $gf$ is an inflation, then so is $f$.
	
	(2) Let $f\in\Mcc(A,B),g\in\Mcc(B,C)$ be any composable  pair of morphisms. If $gf$ is a deflation, then so is $g$.
\end{cond}
\begin{prop}\label{INDE}\citep[Proposition 4.13]{JDP}
	Let $f\in\Mcc(A,B),g\in\Mcc(B,C)$ be any composable  pair of morphisms. We have that
	
	(1) if $gf$ is a $\xi$-inflation, then so is $f$;
	
	(2) if $gf$ is a $\xi$-deflation, then so is $g$.
\end{prop}

\section{Gorenstein derived functors}

\quad~Throughout this section, we always assume that $\mathcal{C}=(\Mcc,\Mbe,\mathfrak{s})$ is 
an extriangulated category
with enough $\xi$-projectives and enough $\xi$-injectives satisfying  Condition(WIC)) and $\xi$ 
is a proper class of $\Mbe$-triangles in $(\Mcc,\Mbe,\mathfrak{s})$.

 \begin{defn}\citep[Definition 3.2]{JZP}
	Let $A$ and $B$ be objects in $\mathcal{C}$.
	
	(1) If we choose a $\xi$-projective resolution $\xymatrix{\mathbf{P}\ar[r]&A}$ of $A$, then for any  integer $n\geqslant 0$, the $\xi$-cohomology groups $\xi \text{xt}_{\mathcal{P}(\xi)}^n(A, B)$ are defined as
	$$
	\xi \text{xt}_{\mathcal{P}(\xi)}^n(A, B)=H^n(\mathcal{C}(\mathbf{P},B)).
	$$
	
	(2) If we choose a $\xi$-injective coresolution $\xymatrix{B\ar[r]&\mathbf{I}}$ of $B$, then for any  integer $n\geqslant 0$, the $\xi$-cohomology groups $\xi \text{xt}_{\mathcal{I}(\xi)}^n(A, B)$ are defined as
	$$
	\xi \text{xt}_{\mathcal{I}(\xi)}^n(A, B)=H^n(\mathcal{C}(A,\mathbf{I})).
	$$
	
	Then there exists an isomorphism $\xi \text{xt}_{\mathcal{P}(\xi)}^n(A, B)\backsimeq \xi \text{xt}_{\mathcal{I}(\xi)}^n(A, B)$, which is denoted by $\xi \text{xt}_{\xi}^n(A,B)$.
	\end{defn}
	
	\begin{lem}\label{LZHL}\citep[Lemma 3.4]{JZP}
	If $\xymatrix{A\ar[r]^x &B\ar[r]^y &C\ar@{-->}[r]^{\delta}&}$ is an $\mathbb{E}$-triangle in $\xi$, then for any objects $X$  in $\mathcal{C}$, we have the following long exact sequences in $\mathbf{Ab}$
	$$
	\xymatrix{
	0\ar[r]&\ext^0(X,A)\ar[r]&\ext^0(X,B)\ar[r]&\ext^0(X,C)\ar[r]&\ext^1(X,A)\ar[r]&\cdots
	}
	$$
	and
	$$
	\xymatrix{
	0\ar[r]&{\ext^0(C,X)}\ar[r]&{\ext^0(B,X)}\ar[r]&{\ext^0(A,X)}\ar[r]&{\ext^1(C,X)}\ar[r]&{\cdots}}
	$$
	For any objects $A$ and $B$, there is always a natural map $\delta: \mathcal{C}\rightarrow\ext^0(A,B)$, which is an isomorphism if $A\in \P$ or $B\in \mathcal{I}(\xi)$.
	\end{lem}
	
	\begin{lem}\label{ZHL}\citep[Lemma 3.5]{JZP}
	Let $M\in\mathcal{C}$ and $G\in\GP$. If 
	$M\in\widehat{\Mcp}(\xi)$ or $M\in\widehat{\mathcal{I}}(\xi)$, then
	 $${\ext^0(G,M)}\backsimeq \mathcal{C}(G,M)\  \text{and}\  \ext^i(G,M)=0 \ \text{for any $i\geqslant 1$}.$$
	\end{lem}
	\begin{defn}\label{PGR}
	\emph{A proper $\Gproj$ resolution }of an object $A$ in $\Mcc$ is
	 a $\xi$-exact 
	complex $$\xymatrix{\cdots\ar[r]&G_1\ar[r]&G_0\ar[r]&A\ar[r]&0} (\text{for short we write by}\ \mathbf{G}\longrightarrow A)$$ 
	such that $G_i\in\GP$ and for any $i$, the relevant $\Mbe$-triangle $\xymatrix{K_{i+1}\ar[r]&G_i\ar[r]&K_i\ar@{-->}[r]&}$ ( set $K_0=A$ ) is $\Mcc(\GP,-)$-exact. This resolution is said to be of length $n$ if $G_n\neq0$ and $G_i=0$ for all $i > n$.
And we say that the complex 
\[
\cdots \longrightarrow G_2	 \longrightarrow G_1 \longrightarrow G_0 \longrightarrow 0
\]	
is the deleted complex for the proper $\Gproj$ resolution $\mathbf{G}\longrightarrow A$, which denoted by $\mathbf{G}^{\sharp}\rightarrow A$
\end{defn}
	
	Dually, for any object $B\in\Mcc$, we can define proper $\xi$-$\mathcal{G}$injective coresolution $B\longrightarrow \mathbf{H}$ and its deleted complex $B\longrightarrow \mathbf{H}_{\sharp}$.
	
	\begin{lem}[Horseshoe Lemma*]\label{HS*}
	Let $A\stackrel{x}{\longrightarrow}B\stackrel{y}{\longrightarrow}C\stackrel{\delta}\dashrightarrow$  be a $\Mcc(\GP,-)$-exact $\Mbe$-triangle in $\xi$. If $A$ has a proper $\Gproj$ resolution $\xymatrix{\mathbf{G_A}\ar[r]&A}$ and $C$ has a proper $\Gproj$ resolution $\xymatrix{\mathbf{G_C}\ar[r]&C}$, then there is a proper $\xi$-$\mathcal{G}$projective resolution $\xymatrix{\mathbf{G_B}\ar[r]&B}$ making the diagram below commutative:
	$$\xymatrix{
	\mathbf{G}_A\ar[d]\ar[r]^{x^{\bullet}}&\mathbf{G}_B\ar[d]\ar[r]^{y^{\bullet}}&\mathbf{G}_C\ar[d]\ar@{-->}[r]&\\
	A\ar[r]^x&B\ar[r]^y&C\ar@{-->}[r]&.
	}
	$$
	Moreover, $\xymatrix{G^n_A\ar[r]^{x^n}&G^n_B\ar[r]^{y^n}&G^n_C\ar@{-->}[r]&}$ is a split $\Mbe$-triangle, i.e. $G^n_B\simeq G^n_A\oplus G^n_C$ for any $n\geqslant0$.
	\end{lem}
	
	\begin{proof}
	The proof is similar to a part of proof for \citep[Theorem 4.16]{JDP}. For the convenience of readers, we give the proof here.
	
	Since $A$ has a proper $\Gproj$ resolution $\xymatrix{\mathbf{G_A}\ar[r]&A}$, then there exists a $\Mcc(\GP,-)$-exact $\mathbb{E}$-triangle $\xymatrix{K_1^A\ar[r]^{t_1^A}&G_0^A\ar[r]^{d_0^A}&A\ar@{-->}[r]^{\delta^A_0}&}$ in $\xi$ with $G_0^A\in\GP$. Similarly, there exists a $\Mcc(\GP,-)$-exact $\mathbb{E}$-triangle $\xymatrix{K_1^C\ar[r]^{t_1^C}&G_0^C\ar[r]^{d_0^C}&C\ar@{-->}[r]^{\delta^C_0}&}$ in $\xi$ with $G_0^C\in\GP$. By Lemma \ref{BH}, there exists following commutative diagram:
	$$
	\xymatrix{
	&K_1^C\ar@{=}[r]\ar[d]&K_1^C\ar[d]^{t_1^C}\\
	A\ar[r]\ar@{=}[d]&M\ar[r]\ar[d]&G_0^C\ar[d]^{d_0^C}\ar@{-->}[r]&\\
	A\ar[r]^x&B\ar@{-->}[d]\ar[r]^y&C\ar@{-->}[r]^{\delta}\ar@{-->}[d]^{\delta^C_0}&.\\
	&&
	}
	$$
	Since $\xymatrix{A\ar[r]^{x}&B\ar[r]^{y}&C\ar@{-->}[r]^{\delta}&}$ is $\Mcc(\GP,-)$-exact, there is a exact sequence
	$$
	0\longrightarrow\Mcc(G_0^C,A)\stackrel{\mathcal{C}(G_0^C,x)}\longrightarrow \Mcc(G_0^C,B)\stackrel{\mathcal{C}(G_0^C,y)}\longrightarrow\Mcc(G_0^C,C)\longrightarrow0
	$$
	in $\mathbf{Ab}$. So we have that $d_0^C$ factors through $y$. This implies $(d_0^C)^*\delta=(d_0^A)_*0=0$ by Lemma \ref{FJ}, hence there is an deflation $d_0^B : G_0^B=: G_0^A \oplus G_0^C \rightarrow B$ which makes the following diagram commutative
	$$
	\xymatrix{G_0^A\ar[r]^{\begin{tiny}\begin{bmatrix}
	1 \\
	0
	\end{bmatrix}\end{tiny}}\ar[d]^{d_0^A}&G_0^B\ar[r]^{\begin{tiny}\begin{bmatrix}
	0&1
	\end{bmatrix}\end{tiny}}\ar[d]^{d_0^B}&G_0^C\ar@{-->}^0[r]\ar[d]^{d_0^C}&\\
	A\ar[r]^x&B\ar[r]^y&C\ar@{-->}[r]^{\delta}&
	}
	$$
	by \citep[Lemma 4.15(2)]{JDP}. Hence there exists an $\Mbe$-triangle $\xymatrix{K_1^B\ar[r]^{t_1^B}&G_0^B\ar[r]^{d_0^B}&B\ar@{-->}[r]^{\delta_0^B}&}$. There exists a commutative diagram
	$$
	\xymatrix{
	K_1^A\ar[r]^{x_1}\ar[d]^{t_1^A}&K_1^B\ar[r]^{y_1}\ar[d]^{t_1^B}&K_1^C\ar@{-->}[r]^{\delta_1}\ar[d]^{t_1^B}&\\
	G_0^A\ar[r]^{\begin{tiny}\begin{bmatrix}
	1 \\
	0
	\end{bmatrix}\end{tiny}}\ar[d]^{d_0^A}&G_0^B\ar[r]^{\begin{tiny}\begin{bmatrix}
	0&1
	\end{bmatrix}\end{tiny}}\ar[d]^{d_0^B}&G_0^C\ar@{-->}^0[r]\ar[d]^{d_0^C}&\\
	A\ar[r]^x\ar@{-->}[d]^{\delta_0^A}&B\ar[r]^y\ar@{-->}[d]^{\delta_0^B}&C\ar@{-->}[r]^{\delta}\ar@{-->}[d]^{\delta_0^C}&\\
	&&&
	}
	$$
	made of $\Mbe$-triangles by \citep[Lemma 4.14]{JDP}. Since $\begin{tiny}\begin{bmatrix}
	1 \\
	0
	\end{bmatrix}\end{tiny}$ and $t_1^A$ are $\xi$-inflations, so is $x_1$ by Proposition \ref{INDE}(1). So the $\Mbe$-triangle
	$\xymatrix{K_1^A\ar[r]^{x_1}&K_1^B\ar[r]^{y_1}&K_1^C\ar@{-->}[r]^{\delta_1}&}$ is isomorphic to an $\Mbe$-triangle in $\xi$ by \citep[Corollary 3.6(3)]{HY}, hence it is an $\Mbe$-triangle in $\xi$. Applying functor $\Mcc(\mathcal{P}(\xi),-)$ to the above commutative diagram, it is easy to prove that the $\Mbe$-triangle $\xymatrix{K_0^B\ar[r]^{t_1^B}&G_0^B\ar[r]^{d_0^B}&B\ar@{-->}[r]^{\delta_0^B}&}$ is $\Mcc(\mathcal{P}(\xi),-)$-exact by a diagram chasing, hence it is an $\Mbe$-triangle in $\xi$ by Lemma \ref{FBZH1}. Applying functor $\Mcc(\GP,-)$ to the above commutative diagram, it is not difficult to show that the $\Mbe$-triangles $\xymatrix{K_0^B\ar[r]^{t_1^B}&G_0^B\ar[r]^{d_0^B}&B\ar@{-->}[r]^{\delta_0^B}&}$ and $\xymatrix{K_1^A\ar[r]^{x_1}&K_1^B\ar[r]^{y_1}&K_1^C\ar@{-->}[r]^{\delta_1}&}$  are $\Mcc(\GP,-)$-exact by a diagram chasing. Proceeding in this manner, one can get a $\Mcc(\GP,-)$-exact $\xi$-exact complex
	$$
	\xymatrix{
	\cdots\ar[r]&G_2^B\ar[r]&G_1^B\ar[r]&G_0^B\ar[r]&B\ar[r]&0
	}
	$$
	which is a proper $\Gproj$ resolution of $B$.
	\end{proof}

	\begin{thm}[Comparison Theorem]
		Let $f:A\rightarrow A'$ be a morphism in $\Mcc$, and let
		$$\xymatrix{
		\mathbf{A}:\cdots \ar[r]&G_2\ar[r]^{d_2}&G_1\ar[r]^{d_1}&G_0\ar[r]^{d_0}&A\ar[r]&0
		}
		$$
		be a $\Gproj$ resolution of $A$,
		$$\xymatrix{
		\mathbf{A'}:\cdots \ar[r]&G'_2\ar[r]^{d'_2}&G'_1\ar[r]^{d'_1}&G'_0\ar[r]^{d'_0}&'A\ar[r]&0
		}
		$$
		be a proper $\Gproj$ resolution of $A'$. Then
		
		(1) there exists a chain map $f^{\bullet}: \mathbf{A}\rightarrow \mathbf{A'}$ making the diagram below commutative:
		$$
		\xymatrix{
		\cdots\ar[r]&G_n\ar[r]^{d_{n}}\ar@{-->}[d]^{f_n}&G_{n-1}\ar[r]\ar@{-->}[d]^{f_{n-1}}&\cdots\ar[r]&G_1\ar[r]^{d_1}\ar@{-->}[d]^{f_1}&G_0\ar[r]^{d_0}\ar@{-->}[d]^{f_0}&A\ar[r]\ar[d]^{f}&0\\
		\cdots\ar[r]&G'_n\ar[r]^{d_n}&G'_{n-1}\ar[r]&\cdots\ar[r]&G'_1\ar[r]^{d'_1}&G'_0\ar[r]^{d'_0}&A'\ar[r]&0
		}
		$$
		
		(2) the morphism $f^{\bullet}$ in $\mathbf{C}(\Mcc)$  constructed in (1) is unique up to homotopy.
		\end{thm}

		\begin{proof}
		(1) Since $\mathbf{A}$ is the $\Gproj$ resolution of $A$, then there exists an $\Mbe$-triangle
		$$\xymatrix{K_{n+1}\ar[r]^{x_{n+1}}&G_n\ar[r]^{y_n}&K_n\ar@{-->}[r]^{\delta_n}&}$$
		in $\xi$ for all $n\geqslant0$ with $x_ny_n=d_n$, $K_0=A$, $y_0=d_0$.
		
		Note that $\mathbf{A'}$ is the proper $\Gproj$ resolution of $A'$, then there exists a $\Mcc(\GP,-)$-exact $\Mbe$-triangle
		$$\xymatrix{K'_{n+1}\ar[r]^{x'_{n+1}}&G'_n\ar[r]^{y'_n}&K'_n\ar@{-->}[r]^{\delta'_n}&}$$
		in $\xi$ for all $n\geqslant0$ with $x'_ny'_n=d'_n$, $K'_0=A'$, $y'_0=d'_0$. So for any integer $n\geqslant 0$ and any $G\in\GP$, one can get a short exact sequence
		$$
		\xymatrix{
		0\ar[r]&\Mcc(G,K'_{i+1})\ar[r]^{\Mcc(G,x'_{i+1})}&\Mcc(G,G'_i)\ar[r]^{\Mcc(G,y'_i)}&\Mcc(G,K'_i)\ar[r]&0}\quad(*)
		$$
		in $\mathbf{Ab}$ for all $i\geqslant0$. Take $G=G_0,i=0$ in $(*)$ and $fd_0\in\Mcc(G_0,A')=\Mcc(G_0,K'_0)$, then there exists a morphism $f_0\in \Mcc(G_0,G'_0)$ such that $$fd_0=\Mcc(G_0,y'_0)f_0=\Mcc(G_0,d'_0)f_0=d'_0f_0.$$
		By $\rm(ET3)^{op}$, there exists a  a morphism $g_1\in \Mcc(K_1,K'_1)$ making the diagram below commutative:
		$$
		\xymatrix{
		K_1\ar[r]^{x_1}\ar@{-->}[d]^{g_1}&G_0\ar[r]^{d_0}\ar[d]^{f_0}&A\ar@{-->}[r]^{\delta_1}\ar[d]^f&\\
		K'_1\ar[r]^{x'_1}&G'_0\ar[r]^{d'_0}&A'\ar@{-->}[r]^{\delta'_1}&.
		}
		$$
		Assume that there exist $f_0\text{,}f_1\text{,}\cdots\text{,}f_n$ and
		 $g_0=f\text{,}g_1\text{,}\cdots\text{,}g_{n}\text{,}g_{n+1}$
		 which make the diagrams below commutative:
		$$
		\xymatrix{
		K_{i+1}\ar[r]^{x_{i+1}}\ar[d]^{g_{i+1}}&G_i\ar[r]^{y_i}\ar[d]^{f_i}&K_i\ar@{-->}[r]^{\delta_i}\ar[d]^{g_i}&\\
		K'_{i+1}\ar[r]^{x'_{i+1}}&G'_i\ar[r]^{y'_i}&K'_i\ar@{-->}[r]^{\delta'_i}&
		}
		$$
		for $i=0\text{,}\cdots\text{,}n$.
		Then one can get that
		$$d'_if_i=x'_1y'_if_i=x'_1g_iy_i=f_{i-1}x_iy_i=f_{i-1}d_i,\  i=1,\cdots n.$$
		Take $G=G_{n+1},i=n+1$ in $(*)$ and $g_{n+1}y_{n+1}\in\Mcc(G_{n+1},K'_{n+1})$, there exists a morphism $f_{n+1}\in\Mcc(G_{n+1},G'_{n+1})$ such that
		$$
		d'_{n+1}f_{n+1}=x'_{n+1}y'_{n+1}f_{n+1}=x'_{n+1}g_{n+1}y_{n+1}=f_nx_{n+1}y_{n+1}=f_nd_{n+1}.
		$$
		By induction, we have proved the existence of the $f^{\bullet}$.
		
		(2) Let $f'^{\bullet}$ be a another chain map from $\mathbf{A}$ to $\mathbf{A'}$. Take $G=G_0,i=0$ in $(*)$ and $d'_0(f_0-f'_0)=fd_0-fd_0=0$, then $f_0-f'_0\in$Ker $(\Mcc(G_0,d'_0))=$ Im $(\Mcc(G_0,x'_1))$. So there exists a morphism $t_0\in\Mcc(G_0,K'_1)$ such that $x'_1t_0=f_0-f'_0$. Take $G=G_0,i=1$ in $(*)$, then there exists a morphism $s_0\in\Mcc(G_0,G'_1)$ such that $t_0=y'_1s_0$ and
		$$f_0-f'_0=x'_1t_0=x'_1y'_1s_0=d'_1s_0.$$
		Let $h_1=f_1-f'_1-s_0d_1\in\Mcc(G_1,G'_1)$, then
		$$
		(\Mcc(G_1,x'_1)\circ\Mcc(G_1,y'_1))(h_1)=x'_1y'_1h_1=x'_1y'_1(f_1-f'_1-s_0d_1)=0.
		$$
		Since $\Mcc(G_1,x'_1)$ is a monomorphism, then $\Mcc(G_1,y'_1))(h_1)=0$ 
		i.e. $$ h_1\in \text{Ker}\ \Mcc(G_1,y'_1)= \text{Im}\ \Mcc(G_1,x'_2).$$
		So there exists a morphism $t_1\in\Mcc(G_1,K'_2)$ such that $h_1=x'_2t_1$.
		
		And take $G=G_1,i=2$ in $(*)$, then there exists a morphism $s_1\in\Mcc(G_1,G'_2)$ with $t_1=\Mcc(G_1,y'_2)(s_1)=y'_2s_1$. So we have
		$$
		d'_2s_1+s_0d_1=x'_2y'_2s_1+s_1d_1=x'_2t_1+s_0d_1=h_1+s_0d_1=f_1-f'_1.
		$$
		Proceeding this manner, one can get $\{s_n\}_{n\geqslant0}$ such that
		$$
		f_0-f'_0=d'_0s_0,\  f_n-f'_n=d'_{n+1}s_n+s_{n-1}d_n,\  n\geqslant 1.
		$$
		This is enough to say that $f^{\bullet}$ is homotopic to $f'^{\bullet}$.
		\end{proof}

		\begin{rem}\label{Homotopy}
			Any two proper $\Gproj$ resolution of any object $A\in\Mcc$ are homotopy equivalent. Dually, Any two proper $\Ginj$ coresolution of any object $A\in\Mcc$ are homotopy equivalent.
			\end{rem}
			
			We use $\mathbf{Res}(\GP)$ ( resp. $\mathbf{Res}( \mathcal{GI}(\xi))$ ) to denote the full subcategory of $\Mcc$ consisting of those objects that have a proper $\Gproj$ resolution ( resp. proper $\Ginj$ coresolution ).
			
			\begin{defn}
			Let $A$, $B$ be two objects in $\Mcc$.

			(1)If $A\in \mathbf{Res}(\GP)$, then $A$ has a proper $\Gproj$ 
			resolution $\xymatrix{\mathbf{G}\ar[r]&A}$. For any $i\geqslant0$, 
			we define $\extgp^n(A,B)$ to be 
			the $n$th cohomology of the induced complex $\Mcc({G}^{\sharp},B)$ i.e.
			\[
				\extgp^n(A,B):=H^n(\Mcc({G}^{\sharp},B)),n\geqslant 0.
			\]

			(2)If $B\in \mathbf{Res}(\GI)$, then $B$ has a proper $\Ginj$ coresolution $\xymatrix{{B}\ar[r]&\mathbf{H}}$. For any $i\geqslant0$, 
			we define $\extgi^n(A,B)$ to be the $n$th cohomology of the induced complex $\Mcc(A,\mathbf{H}_{\sharp})$ i.e.
			\[
				\extgi^n(A,B):=H^n(\Mcc(A,{H}_{\sharp})),n\geqslant 0.
			\]	
		\end{defn}
			
			It is easy to see that $\extgp^n(-,-)$ and $\extgi^n(-,-)$ are well defined by Remark \ref{Homotopy}.
			
			\begin{prop}
			Let $X$ be an object in $\Mcc$, then  
			
			(1)if $G\in\GP$, then $\extgp^i(G,X)=0$ for any $i\geqslant 1$.
			
			(2)if $Q\in\GI$, then $\extgi^i(X,Q)=0$ for any $i\geqslant 1$.
			\end{prop}
\begin{proof}
	(1) In fact, the $\xi$-exact complex
	\[\cdots\longrightarrow G\stackrel{1}\longrightarrow G\stackrel{0}\longrightarrow G\stackrel{1}\longrightarrow G\stackrel{0}\longrightarrow G\stackrel{1}\longrightarrow G\longrightarrow0
	\]
	is a proper $\Gproj$ resolution of $G$, and its deleted complex induced a complex 
	\[0\rightarrow \Mcc(G,X)\stackrel{0}\rightarrow \Mcc(G,X)\stackrel{0}
	\rightarrow \Mcc(G,X)\stackrel{1}\rightarrow \Mcc(G,X)\stackrel{0}\rightarrow
	 \Mcc(G,X)\stackrel{1}\rightarrow \Mcc(G,X)\rightarrow\cdots
	\]
	in $\mathbf{Ab}$. Then (1) is obvious by definition.
	
	(2) It is the dual of (1).

\end{proof}

			\begin{prop} Let $A$ and $B$ be any objects in $\Mcc$, then
			
				(1) if $A\in\widehat{\mathcal{P}}(\xi)$, then  $\extgp^i(A,B)\backsimeq \ext^n(A,B)$, $n\geqslant 0$;
			
			(2) if $B\in\widehat{\mathcal{I}}(\xi)$, then $\extgi^i(A,B)\backsimeq \ext^n(A,B)$, $n\geqslant 0$.
			\end{prop}
			
			\begin{proof}
			Since (2) is the dual of (1), so we only need to prove (1).
			
			Let $\xi$-pd$A=n$, then there exists a $\xi$-projective resolution of $A$ with length $n$:
			$$
			\xymatrix
			{
			 P_n\ar[r]&P_{n-1}\ar[r]&\cdots\ar[r]&P_1\ar[r]&P_0\ar[r]&A\ar[r]&0.\qquad (**)
			}
			$$
			For the relevant $\mathbb{E}$-triangle $\xymatrix{ K_{i+1}\ar[r]&P_i\ar[r]&K_i\ar@{-->}[r]&}$ ( set $K_0=A$ ) in $\xi$, all items have finite $\xi$-projective dimension.
			Let $G$ be any $\Gproj$ objects of $\Mcc$. There exists a short exact sequence
			$$0\rightarrow \Mcc(G,K_{i+1})\rightarrow \Mcc(G,P_i) \rightarrow \Mcc(G,K_{i})\rightarrow 0 $$
			by Lemma \ref{ZHL}. This implies that the $\xi$-projective resolution of
			 $A$ is $\Mcc(\GP,-)$-exact, and indeed a proper $\Gproj$ resolution. 
			 So we have $$\extgp^n(A,B)\backsimeq \ext^n(A,B)$$
			by definition.
			\end{proof}
			\begin{cor}\label{JYT}
				Let $A$ and $B$ be any objects in $\Mcc$. 
			
				(1) If $A\in\widehat{\mathcal{P}}(\xi)$, there exists a $\Mcc(-,\GI)$-exact proper $\Gproj$ resolution of $A$;  
			
			(2) If $B\in\widehat{\mathcal{I}}(\xi)$,  there exists a $\Mcc(\GP,-)$-exact proper $\mathcal{G}$injective coresolution of $B$.
		\end{cor}

			\begin{thm}
			Let $\xymatrix{A\ar[r]^x&B\ar[r]^y&C\ar@{-->}[r]^{\delta}&}$ be a $\Mcc(\GP,-)$-exact $\Mbe$-triangle in $\xi$ with $A\in \mathbf{Res}(\GP)$ and $C\in\mathbf{Res}(\GP)$, then for any object $X\in\Mcc$, there is always a natural connecting homomorphism
			$$
			\xymatrix{
			\partial^n_X : \extgp^n(A,X)\ar[r]&\extgp^{n+1}(C,X),\ n\geqslant0
			}
			$$
			making the following sequence
			\[
			0\rightarrow {\extgp^0(C,X)}\rightarrow {\extgp^0(B,X)}\rightarrow{\extgp^0(A,X)}\stackrel{\partial_X^0}\rightarrow{\extgp^1(C,X)}\rightarrow\cdots
			\]
			exact.
			\end{thm}
			
			\begin{proof}
			By Lemma \ref{HS*}, there is  a short exact sequence of complexs
			$$ 0\longrightarrow \mathbf{\hat{G}}^A\longrightarrow \mathbf{\hat{G}}^B\longrightarrow \mathbf{\hat{G}}^C\longrightarrow 0
			$$
			where $\mathbf{\hat{G}}^A, \mathbf{\hat{G}}^B$ and
			 $\mathbf{\hat{G}}^C$ are the deleted complexs of proper $\Gproj$ 
			 resolution for $A, B$ and $C$ respctively with $G_n^B=G_n^A\oplus G_n^C,\ n\geqslant0$.
			
			Since $\Mcc(-,X)$ is a additive functor for any $X\in\Mcc$, then we have the following short exact sequence:
			$$ 0\longrightarrow \Mcc(\mathbf{\hat{G}}^C,X)\longrightarrow  \Mcc(\mathbf{\hat{G}}^B,X)\longrightarrow  \Mcc(\mathbf{\hat{G}}^A,X)\longrightarrow 0.
			$$
			By long exact sequence theorem for homology of complex, there is always a natural connecting homomorphism:
			$$
			\xymatrix{
			&\partial^n_X : H_n(\Mcc(\mathbf{\hat{G}}^A,X))\ar@{=}[d]\ar[r]&H_{n+1}(\Mcc(\mathbf{\hat{G}}^C,X))\ar@{=}[d]&\\
			&\extgp^n(A,X)&\extgp^{n+1}(C,X)
			}
			$$
			and $ H_{-1}(\Mcc(\mathbf{\hat{G}}^A,X))=0$. So we completed the proof.
			\end{proof}

			Dually, we have the following result.
			\begin{thm}
			Let $\xymatrix{A\ar[r]^x&B\ar[r]^y&C\ar@{-->}[r]^{\delta}&}$ be a $\Mcc(-,\GI)$-exact $\Mbe$-triangle in $\xi$ with $A\in \mathbf{Res}(\GI)$ and $C\in\mathbf{Res}(\GI)$, then for any object $X\in\Mcc$, there is always a natural connecting homomorphism
			$$
			\xymatrix{
			\partial^n_X : \extgp^n(X,C)\ar[r]&\extgp^{n+1}(X,A),\ n\geqslant0
			}
			$$
			making the following sequence
			\[
			0\rightarrow \extgi^0(X,A)\rightarrow \extgi^0(X,B)\rightarrow\extgi^0(X,C)
			\stackrel{\partial_X^0}\rightarrow\extgi^1(X,A)\rightarrow\cdots
			\]
			exact.
			\end{thm}

			\begin{defn}
			Let $A$ be an object in $\Mcc$.  The $\Mbe$-triangle  $\xymatrix{K\ar[r]&G\ar[r]&A\ar@{-->}[r]&}$ in $\xi$
			is called a weak $\GP$-approximation of $A$, if $G\in\GP$ and $\xi$-pd$K<\infty$. Furthermore, if the above $\Mbe$-triangle
			is $\Mcc(\GP,-)$-exact, then we say it is a $\GP$-approximation of $A$.
		
			Dually, the $\Mbe$-triangle  $\xymatrix{A\ar[r]^{f}&Q\ar[r]^{g}&L\ar@{-->}[r]^{\delta}&}$ in $\xi$
			is called a weak $\GI$-coapproximation of $A$, if $Q\in\GI$ and $\xi$-id$L<\infty$. Furthermore, if the above $\Mbe$-triangle
			is $\Mcc(-,\GI)$-exact, then we say it is a $\GI$-coapproximation of $A$.			
			\end{defn}
			
\begin{rem}
	If $C$ is a triangulated category or exact category, then  the weak $\GP$-approximation 
	is equivalent to $\GP$-approximation and weak $\GI$-coapproximation is  equivalent to $\GI$-approximation for any object $A\in\Mcc$
\end{rem}

\begin{rem}
	Let $A$ be any object in $\Mcc$. 

	(1) If $\Gpd A<\infty$, there is a  weak $\GP$-approximation of $A$;

	(2) If $\mathcal{G}$id $A<\infty$,  there is a  weak $\GI$-coapproximation of $A$.
\end{rem}

\begin{defn}
	Let $A$ be any object in $\Mcc$. The   $\Mbe$-triangle  $\xymatrix{K\ar[r]^{f}&G\ar[r]^{g}&A\ar@{-->}[r]^{\delta}&}$ in $\xi$	
is called a $\xi$-successor of $A$ if

(1) it is a  $\GP$-approximation of $A$;

(2) for any $M\in\GI$, the morphism $\Mcc(g,M)$ is always monomorphism. 

\noindent Dually, the   $\Mbe$-triangle  $\xymatrix{A\ar[r]^{f}&Q\ar[r]^{g}&L\ar@{-->}[r]^{\delta}&}$ in $\xi$	
is called a $\xi$-predecessor of $A$ if

(1) it is a  $\GI$-coapproximation of $A$;

(2) for any $N\in\GP$, the morphism $\Mcc(N,f)$ is always epimorphism.

\end{defn}

We use $\mathcal{GP}^*(\xi)$ (resp.  $\mathcal{GI}^*(\xi)$) to denote the full subcategory of $\Mcc$ whose objects exist $\xi$-successor
 (resp. $\xi$-predecessor).
			
			\begin{rem}
			$\GP$ is the full subcategory of $\mathcal{GP}^*(\xi)$, and $\GI$ is the full subcategory of $\mathcal{GI}^*(\xi)$.
			\end{rem}
			
			\begin{proof}
			For any $G\in\GP$, $Q\in\GI$, it is easy to see that the
			 $\mathbb{E}$-triangles
			  $$\xymatrix{0\ar[r]&G\ar[r]^1&G\ar@{-->}[r]&}\ \text{and}\ \xymatrix{Q\ar[r]^1&Q\ar[r]&0\ar@{-->}[r]&}$$
			   are the $\xi$-successor of $G$ and $\xi$-predecessor of $Q$ respctively.
			\end{proof}
			
			\begin{thm}\label{GEXT}
			Let $A$ and $B$ be two objects in $\Mcc$ with 
			$A\in \mathcal{GP}^*(\xi)$, $B\in \mathcal{GI}^*(\xi)$. 
			Then $\extgp^n(A,B)\backsimeq \extgi^n(A,B)$ for any $n\geqslant 0$.
			\end{thm}
			
			\begin{proof}
			Since  $A\in \mathcal{GP}^*(\xi)$, there exists a $\GP$-approximation  $$\xymatrix{K_1\ar[r]^{f}&G_0\ar[r]^{g}&A\ar@{-->}[r]&}.$$
			For any $Q\in\GI$, there is a $\mathbb{E}$-triangle $\xymatrix{Q'\ar[r]^{x}&I\ar[r]^{y}&Q\ar@{-->}[r]&}$ in $\xi$, such that $Q'\in\GI$, $I\in\mathcal{I}(\xi)$. Note that $\xi$-pd$K_1<\infty$, by Lemma \ref{LZHL} and the dual of Lemma \ref{ZHL}, we have a short exact sequence
			$$
			0\rightarrow\Mcc(K_1,Q')\rightarrow\Mcc(K_1,I)\rightarrow\Mcc(K_1,Q)\rightarrow 0.
			$$
			Then for any morphism $\alpha\in\Mcc(K_1,Q)$, there exists a morphism $\beta\in\Mcc(K_1,I)$ such that $\alpha=y\beta$. Since $I\in\mathcal{I}(\xi)$, we can get a short exact sequence
			$$
			0\rightarrow\Mcc(A,I)\rightarrow\Mcc(G_0,I)\rightarrow\Mcc(K_1,I)\rightarrow 0.
			$$
			  So there is a morphism $\gamma\in\Mcc(G_0,I)$ such that $\beta=\gamma f$. Thus, we have $$\alpha=y\beta=y\gamma f=\Mcc(f,Q)(y\gamma).$$ This implies that $\Mcc(f,Q)$ is a epimorphism. And $A$ is in $\mathcal{GP}^*(\xi)$, then $\Mcc(g,Q)$ is a monomorphism. Therefore, the sequence
			$$
			0\rightarrow\Mcc(A,Q)\rightarrow\Mcc(G_0,Q)\rightarrow\Mcc(K_1,Q)\rightarrow 0
			$$
			is exact. There is a $\mathbb{E}$-triangle $\xymatrix{K_2\ar[r]&P_1\ar[r]&K_1\ar@{-->}[r]&}$ in $\xi$ with $P_1\in\P$ and $\xi$-pd$K_2<\infty$ since $\xi$-pd$K_1<\infty$. It follows from Lemma \ref{LZHL}, Lemma \ref{ZHL} and the dual of Lemma \ref{ZHL} that the sequences
			$$
			0\rightarrow\Mcc(G,K_2)\rightarrow\Mcc(G,P_1)\rightarrow\Mcc(G,K_1)\rightarrow 0
			$$
			and
			$$
			0\rightarrow\Mcc(K_1,Q)\rightarrow\Mcc(P_1,Q)\rightarrow\Mcc(K_2,Q)\rightarrow 0
			$$
			are exact for any $G\in\GP$ and $Q\in\GI$. Continuing this process, one can get a proper $\Gproj$ resolution of $A$ which is $\Mcc(-,\GI)$-exact.
			
			Dually, for $B\in\mathcal{GI}^*(\xi)$, there exists a proper $\Ginj$ coresolution of $B$ which is $\Mcc(\GP,-)$-exact. By \citep[Proposition 2.3]{EO}, the desired isomorphism
			$$
			\extgp^n(A,B)\backsimeq\extgi^n(A,B),\  n\geqslant0
			$$
			holds.
			\end{proof}
			
       If $A\in \mathcal{GP}^*(\xi)$, $B\in \mathcal{GI}^*(\xi)$, we let 
			$\gext^n(A,B)=\extgp^n(A,B)\backsimeq \extgi^n(A,B)$, and we say 
			$\gext^n(-,-)$ is the $\xi$-Gorenstein derived functor of $\Mcc(-,-)$.
			
			At the end of this section, we give some applications of  derived functors and 
			Gorenstein derived functors.		
				
			\begin{prop}
				Let $A$ be  any object in $\Mcc$. If $\Gpd A\leqslant 1$ and 
				$\ext^1(A,\P)=0$, then $A$ is a $\Gproj$ object.
			   \end{prop}
			   \begin{proof}
			   Since $\Gpd A\leqslant 1$, there exists an $\Mbe$-triangle 
			   $\xymatrix{K\ar[r]&G\ar[r]&A\ar@{-->}[r]&}$ in $\xi$, with $G\in\GP$
			   and $K\in\GP$. Note that $\ext^1(A,\P)=0$, by \citep[Lemma 3.6]{JZP}, $A\in\GP$.
			   \end{proof}
			   \begin{thm}
				Let $A$ be  any object in $\Mcc$. For any integer $n\geqslant 1$, if 
				$\ext^n(A,A)=0$, then $A\in\nSGP$ (see \cite[Definition 4.3]{He}) if and only if $A\in\P$.
			   \end{thm}
			   \begin{proof}
			   Note that $\P\subseteq\nSGP$, if $A\in\P$,  then $A\in\nSGP$ is obvious. 
			   
			   If $A\in\nSGP$, there exists a complete $\P$-exact complex:
			   \[  0\longrightarrow A\longrightarrow P_{n-1}\longrightarrow P_{n-2}\longrightarrow\cdots\longrightarrow P_1\longrightarrow P_0\longrightarrow A\longrightarrow0
			   \]
			   where $P_i\in\P,\ 0\leqslant i \leqslant n-1$. Thus for any $0\leqslant i\leqslant n-1$, we have 
			   $\mathcal{C}(-,\P)$-exact $\xi$-resolution $\Mbe$-triangle
			   $\xymatrix{ K_{i+1}\ar[r]&P_i\ar[r]&K_i\ar@{-->}[r]&}$, $K_{n}=K_0=A$. So there is a exact sequence 
			   $$
			   \ext^m(P_{i},A)\longrightarrow\ext^m(K_{i+1},A)\longrightarrow\ext^{m+1}(K_{i},A)\longrightarrow\ext^{m+1}(P_{i},A).
			   $$
			   Note that for any $m\geqslant 1$, $\ext^m(P_{i},A)=0$,  So $\ext^m(K_{i+1},A)\backsimeq\ext^{m+1}(K_{i},A)$. 
			   By dimension shifting $$\ext^1(K_{n-1},A)=\ext^n(K_0,A)=\ext^n(A,A)=0.$$
			   Then for $\Mbe$-triangle $\xymatrix{A\ar[r]&P_{n-1}\ar[r]&K_{n-1}\ar@{-->}[r]&}$, we get the short exact sequence
			   $$
			   0\longrightarrow \Mcc(K_{n-1},A)\longrightarrow\Mcc(P_{n-1},A) \longrightarrow \Mcc(A,A)\longrightarrow0.
			   $$
			   By Lemma \ref{FJ},This is to say $P_{n-1}\backsimeq A\oplus K_{n-1}$, so $A\in\P$.
			   \end{proof}

			   \begin{defn}
				Let $\mathcal{X}$ be a class of some object in $\Mcc$, and set
				\begin{equation*}
				 \begin{split}
				&\mathcal{X}^{\bot}=\{B\in\Mcc~|~\ext^n(X,B)=0,\forall n\geqslant1,\forall X\in\mathcal{X}\},\\
				&^{\bot}\mathcal{X}=\{A\in\Mcc~|~\ext^n(A,X)=0,\forall n\geqslant1,\forall X\in\mathcal{X}\}.
				 \end{split}
				\end{equation*}
				\end{defn}
				\begin{rem}
				$\GP\subseteq$$^{\bot}\widehat{\mathcal{P}}(\xi)$ and $\GP\subseteq$$^{\bot}\widehat{\mathcal{I}}(\xi)$, 
				$\GI\subseteq\widehat{\mathcal{P}}(\xi)^{\bot}$ and $\GI\subseteq\widehat{\mathcal{I}}(\xi)^{\bot}$.
				\end{rem}
				
				\begin{rem}
					Let $A$ be  any object in $\Mcc$. According to \citep[Theorem 3.8]{JZP}, if $A\in\widehat{\mathcal{GP}}(\xi)$, 
					then
				\begin{equation*}
				 \begin{split}
				\Gpd A&= \sup\{n \in  \mathbb{N}~|~\exists P\in\P~\text{such that }~\ext^n(A, P)\neq 0\}\\
				&= \sup\{n \in \mathbb{N}~|~\exists P\in\widehat{\mathcal{P}}(\xi)~\text{such that}~\ext^n(A, P)\neq 0\}.
				 \end{split}
				\end{equation*}
				So we have 
				\end{rem}
				\begin{prop}\label{XJK}
				$\GP=\widehat{\mathcal{GP}}(\xi)\bigcap$$^{\bot}\widehat{\mathcal{P}}(\xi)=\widehat{\mathcal{GP}}(\xi)\bigcap$$^{\bot}\P$.
				\end{prop}
				\begin{prop}\label{MD}
				Suppose that $\sup\{\xi\text{-}\mathcal{G}\text{id}M \  |\  \forall M\in\Mcc\}<\infty$, then for any object $M$ in 
				$\mathbf{Res}(\GP)\bigcap \widehat{\mathcal{GP}}(\xi) $, the following are equivalent:

				(1) $\Gpd M\leqslant n$.
				
				(2) $\extgp^{n+i}(M,N)=0$, ~$\forall i\geqslant1$,~ $\forall N\in\mathcal{C}$.
				
				(3) $\extgp^{n+i}(M,P)=0$, ~$\forall i\geqslant1$,~$\forall P\in \P$.
				\end{prop}
				
				\begin{proof}
				(1) $\Rightarrow$ (2) $\Rightarrow$ (3) are obvious, we only need to prove (3) $\Rightarrow$ (1).

				Assume the complex $\cdots\longrightarrow G_1\longrightarrow G_0\longrightarrow M$
				is a proper $Gproj$ resolution of $M$ with $\Mcc(\GP,-)$-exact resolution $\Mbe$-triangle 
				$\xymatrix{K_{j+1}\ar[r]&G_j\ar[r]&K_i\ar@{-->}[r]&}$ (Set $K_0=M$) for any integer $j\geqslant0$.
 Let $P$ be any object in $\P$.  There is a exact sequence
				$$\xymatrix@C=.5cm{
				\cdots\ar[r]&\ext^m(G_j,P)\ar[r]& \ext^m(K_{j+1},P)\ar[r]& \ext^{m+1}(K_j,P)\ar[r]&\ext^{m+1}(G_j,P)\ar[r]&\cdots
				}
				$$ 
				by Lemma \ref{ZHL},
				where $\ext^m(G_j,P)=\ext^{m+1}(G_j,P)=0$, so
				$$ \ext^m(K_{j+1},Q)= \ext^{m+1}(K_j,Q), \  m>0.$$
				By dimension shifting, for any $i\geqslant 1$, $ \ext^i(K_{n},P)\backsimeq \ext^{n+i}(M,P)$.  
				Since $P\in\P$, by \citep[Proposition 4.6]{JZP}, $\xi$-id$P=\xi\text{-}\mathcal{G}$id$P<\infty$.Thus 
				$$\ext^{n+i}(M,P)=\gext^{n+i}(M,P)=0.$$
				i.e. $\ext^i(K_n,P)=\gext^{n+i}(M,P)=0,\ i\geqslant1$. Then  $K_n\in\GP$ by Proposition \ref{XJK}. 
				It is enough to say $\Gpd M\leqslant n$.
				\end{proof}
				
				Dually, there is following result.

				\begin{prop}\label{TMD}
					Suppose that $\sup\{\xi\text{-}\mathcal{G}\text{pd}M \  |\  \forall M\in\Mcc\}<\infty$, then for any object $M$ in 
					$\mathbf{Res}(\GI)\bigcap \widehat{\mathcal{GI}}(\xi) $, the following are equivalent:

				(1) $\xi\text{-}\mathcal{G}\text{id} N\leqslant n$.
				
				(2) $\extgi^{n+i}(M,N)=0$, ~$\forall i\geqslant1$,~$\forall M\in\mathcal{C}$.
				
				(3) $\extgi^{n+i}(I,N)=0$, ~$\forall i\geqslant1$,~$\forall I\in \mathcal{I}(\xi)$.
				\end{prop}			
				
				We Set $\Mbe_{\xi}:=\Mbe|_{\xi}$, that is  for any $A$, $B\in\Mcc$,Let
				\[
				\Mbe_{\xi}(C,A) = \{\delta\in \Mbe(C,A)~|~\exists \xymatrix{A\ar[r]^x&B\ar[r]^y&C\ar@{-->}[r]^{\delta}&}\in\xi\}\]
				and $\mathfrak{s}_{\xi} := \mathfrak{s}|_{\Mbe_{\xi}}$. By \citep[Theorem 3.2]{JDP}, 
				$(\mathcal{C},\Mbe_{\xi} , \mathfrak{s}_{\xi})$ is also a extriangulated category.
				Consider a part of $\Mbe$-triangle in $\xi$  which lie  on the subcategory $\GP$, and set
				$\xi_{\GP}:=\xi|_{\GP}$, i.e. for an $\Mbe$-triangle $\xymatrix{A\ar[r]&B\ar[r]&C\ar@{-->}[r]&}$ in $\xi$
				$$\xymatrix{A\ar[r]&B\ar[r]&C\ar@{-->}[r]&}\in\xi_{\GP}\Longleftrightarrow A, B,C\in\GP.$$
				Set $\Mbe_{\GP}:=\Mbe_{\xi}|_{\GP}$, i.e.
				$$
				\Mbe_{\GP}(C,A) = \{\delta\in \Mbe(C,A)~|~\exists \xymatrix{A\ar[r]^x&B\ar[r]^y&C\ar@{-->}[r]^{\delta}&}\in\xi\text{~and~}A,B,C\in\GP\}$$
				and $\mathfrak{s}_{\GP} := \mathfrak{s}_{\xi}|_{\Mbe_{\GP}}$.
				\begin{lem}
				$(\GP,\Mbe_{\GP},\mathfrak{s}_{\GP})$is a extriangulated category.
				\end{lem}

				\begin{proof}
					$(\rm{ET1})$, $(\rm{ET2})$,$(\rm{ET3})$ and $(\rm{ET3})^{\mathrm{op}}$ are obvious.  We only 
					prove the $(\rm{ET4})$ and $(\rm{ET4})^{\mathrm{op}}$. Let $\delta\in\Mbe_{\GP}(D,A)$ and 
					 $\delta'\in\Mbe_{\GP}(F,B)$ are $\Mbe_{\GP}$-extensions respectively realized by
					$$\xymatrix{A\ar[r]^f&B\ar[r]^{f'}&D}\text{and}\xymatrix{B\ar[r]^g&C\ar[r]^{g'}&F}.$$
					 Since $(\mathcal{C},\Mbe_{\xi} , \mathfrak{s}_{\xi})$ is a extriangulated category, 
					 there exist a commutative diagram in $\Mcc$ with  $E\in\Mcc$
					$$
					\xymatrix{A\ar[r]^f\ar@{=}[d]&B\ar[r]^{f'}\ar[d]_g&D\ar[d]^d\\
					A\ar[r]^h&C\ar[r]^{h'}\ar[d]_{g'}&E\ar[d]^e\\
					&F\ar@{=}[r]&F
					}
					$$
					 and a $\Mbe_{\xi}$-extension $\delta''\in\Mbe_{\xi}(E,A)$ realized 
					 by $\xymatrix{A\ar[r]^h&C\ar[r]^{h'}&E}$ 
					 such that 
					
					$(\textrm{i})$ $D\stackrel{d}{\longrightarrow}E\stackrel{e}{\longrightarrow}F$ realizes $f'_*\delta'$,
					
					$(\textrm{ii})$ $d^*\delta''=\delta$,
					
					$(\textrm{iii})$ $f_*\delta''=e^*\delta$.
					
					\noindent Note that $\xymatrix{D\ar[r]^d&E\ar[r]^{e}&F\ar@{--}[r]^{f'_*\delta'}&}$ belongs to 
					$\xi$ with  $D, F\in\GP$. So $E\in\GP$,  then $\delta''\in\Mbe_{\GP}(E,A)$, $f'_*\delta'\in\Mbe_{\GP}(F,D)$, i.e. 
					$(\rm{ET4})$ is satisfied. Similarly,$(\rm{ET4})^{\mathrm{op}}$ holds.
					
					Thus $(\GP,\Mbe_{\GP},\mathfrak{s}_{\GP})$ is a extriangulated category.
					\end{proof}

					\begin{lem}
						$\xi_{\GP}$ is a proper class of $(\GP,\Mbe_{\GP},\mathfrak{s}_{\GP})$.
						\end{lem}
						\begin{proof}
						It is easy to check $\xi_{\GP}$ satisfies (1) and (2) of Definition \ref{ZL}. Now we show $\xi_{\GP}$ is saturated.
						
						In the diagram of ~\ref{BH}(1)
						\[
						\xymatrix{
						&{A_2}\ar[d]_{m_2}\ar@{=}[r]&{A_2}\ar[d]^{x_2}\\
						{A_1}\ar@{=}[d]\ar[r]^{m_1}&M\ar[r]^{e_1}\ar[d]_{e_2}&{B_2}\ar[d]^{y_2}\\
						{A_1}\ar[r]^{x_1}&{B_1}\ar[r]^{y_1}&C
						}
						\]
						When $\mathbb{E}$-triangle $\xymatrix{A_2\ar[r]^{x_2}&B_2\ar[r]^{y_2}&C\ar@{-->}[r]^{\delta_2}&}$ and 
						$\xymatrix{A_1\ar[r]^{m_1}&M\ar[r]^{m_1}&B_2\ar@{-->}[r]^{y_2^*\delta_1}&}$ both belong to $\xi_{\GP}$, 
						we claim the $\mathbb{E}$-triangle $\xymatrix{A_1\ar[r]^{x_1}&B_1\ar[r]^{y_1}&C\ar@{-->}[r]^{\delta_1}&}$
						is in $\xi_{\GP}$. In fact, regard the diagram as a commutative diagram in $(\mathcal{C},\Mbe_{\xi} , \mathfrak{s}_{\xi})$ 
						 Since $\xymatrix{A_1\ar[r]^{x_1}&B_1\ar[r]^{y_1}&C\ar@{-->}[r]^{\delta_1}&}$ is in $\xi$ with 
						 $A_1,C\in\GP$, so $B_1\in\GP$. Thus $\mathbb{E}$-triangle
						 $\xymatrix{A_1\ar[r]^{x_1}&B_1\ar[r]^{y_1}&C\ar@{-->}[r]^{\delta_1}&}$ belongs to $\xi_{\GP}$.
						Then we can say $\xi_{\GP}$ is a proper class of the  extriangulated category $(\GP,\Mbe_{\GP},\mathfrak{s}_{\GP})$.
						\end{proof}
						
						We denoted $(\GP,\Mbe_{\GP},\mathfrak{s}_{\GP})$ by $\GP$ for short.
						
						\begin{thm}\label{TF}
						 If $M$ is  $\xi$-projective in $\Mcc$ , then $M$ is  $\xi_{\GP}$-injective in  $\GP$.
						\end{thm}
						\begin{proof}
						For any $\Mbe$-triangle $\xymatrix{A\ar[r]&B\ar[r]&C\ar@{-->}[r]&}$$\in\xi_{\GP}$, There exists long exact sequence
						$$
						0\longrightarrow\ext^0(C,M)\longrightarrow\ext^0(B,M)\longrightarrow\ext^0(A,M)\longrightarrow\ext^1(C,M).
						$$
						Note that $A$, $B$, $C$ are $\Gproj$ and $M\in\P$, by \ref{ZHL}, we have short exact sequence
						$$
						0\longrightarrow\Mcc(C,M)\longrightarrow\Mcc(B,M)\longrightarrow\Mcc(A,M)\longrightarrow0.
						$$
						So $M$ is $\xi_{\GP}$-injective in $\GP$.
						\end{proof}
						\begin{thm}\label{ZM}
						If $M$ is $\xi_{\GP}$-projective in $\GP$, then $M$ is $\xi$-projective in $\Mcc$.
						\end{thm}
\begin{proof}
							Since $\Mcc$ has enough $\xi$-projective objects, there exists $P\in\P$ such that 
							$\Mbe$-triangle $$\xymatrix{K\ar[r]&P\ar[r]&M\ar@{-->}[r]&}$$
							belongs to $\xi$. Note that $P$ and $M$ are both $\Gproj$, then so is $K$. 
							Then the $\Mbe$-triangle $\xymatrix{K\ar[r]&P\ar[r]&M\ar@{-->}[r]&}$ is in $\xi_{\GP}$. 
							Since $M$ is  $\xi_{\GP}$-projective in $\GP$, there is a short exact sequence
							$$
							0\longrightarrow\Mcc(M,K)\longrightarrow\Mcc(M,P)\longrightarrow\Mcc(M,M)\longrightarrow0.
							$$
							By Lemma \ref{FJ}, we get $P\backsimeq K\oplus M$, so $M\in\P$.
							\end{proof}
							\begin{cor}\label{GYT}
							All the $\xi_{\GP}$-projective objects are $\xi_{\GP}$-injective objects in $\GP$.
							\end{cor}
							\begin{defn}
A extriangulated category $\Mcc$ with proper class $\xi$ is called a $\xi$-Frobenius category, if there are enough $\xi$-projective 
objects and $\xi$-injective objects with $\P=\mathcal{I}(\xi)$.

							\end{defn}
							\begin{thm}
							If $\Mcc$is a $\xi$-Frobenius category, then so is $\GP$.
							\end{thm}
							\begin{proof}
							First, We claim $\GP=\GI$ when If $\Mcc$is a $\xi$-Frobenius category.
							 Since $\P=\mathcal{I}(\xi)$, we only need to prove any complete $\xi$-projective resolution is a 
							 completed $\xi$-injective coresolution and any complete $\xi$-injective coresolution is a 
							 completed $\xi$-projective resolution.
							 Let
							$$
							\mathbf{P}: \xymatrix{ \cdots\ar[r]&P_1\ar[r]^{d_1}&P_{0}\ar[r]^{d_0}&P_{-1}\ar[r]&{\cdots}\ ,}
							$$
							is a completed $\xi$-projective resolution. For any integer $n$, there exists a $\Mcc(-,\P)$-exact $\xi$-resolution
							$\Mbe$-triangle $\xymatrix{K_{n+1}\ar[r]&P_n\ar[r]&K_n\ar@{-->}[r]&}$. By Lemma\ref{LZHL},
							we can get short exact sequence
							$$
							0\longrightarrow\Mcc(I,K_{n+1})\longrightarrow\Mcc(I,P_{n})\longrightarrow\Mcc(I,K_{n})\longrightarrow0.
							$$
							So $\mathbf{P}$ is a completed $\xi$-injective coresolution. Similarly, we can get any complete $\xi$-injective coresolution is a 
							completed $\xi$-projective resolution.

							Second, by Corollary \ref{GYT} and its dual, we have that  in category $\GP=(\GP,\Mbe_{\GP},\mathfrak{s}_{\GP})$, 
							 all the $\xi_{\GP}$-projective objects are $\xi_{\GP}$-injective  and 
							 in category $\GI=(\GI,\Mbe_{\GI},\mathfrak{s}_{\GI})$, all the $\xi_{\GI}$-injective objects are 
							 $\xi_{\GI}$-projective.
							Since $\GP=\GI$, the theorem is proved.
							\end{proof}
							
							On the contrary, by Theorem\ref{TF} and \ref{ZM}, we get 
							\begin{thm}
							If $\GP$ is a $\xi_{\GP}$-Frobenius category, then the $\xi$-projective objects in $\Mcc$ are consistent with the $\xi_{\GP}$-projective objects in $\GP$.
							\end{thm}

{\small

}

\begin{thebibliography}{20}
	\bibitem{JS} J. Asadollahi, Sh. Salarian. Gorenstein objects in triangulated categories. Journal of Algebra. 281, 264-286, 2004.

	\bibitem{JS1} J. Asadollahi, Sh. Salarian. Tatecohomology and Gorensteinness for triangulated categories. Journal of Algebra. 299, 480–502, 2006.
	
	\bibitem{AB} M. Auslander, M. Bridger. Stable module theory. American Mathematical Society. Vol. 94, 1969.	
	
	\bibitem{BEL} A. Beligiannis. Relative homological algebra and purity in triangulated categories. Journal of Algebra. 227(1), 268-361, 2000.
	
	\bibitem{BEL1} A. Beligiannis. The homological theory of contravariantly finite subcategories: Auslander-Buchweitz contexts, Gorenstein categories and (co-) stabilization. Communications in Algebra. 28, 4547-4596, 2000.
	
	
	\bibitem{CH} L.W. Christensen. Gorenstein dimensions. Lecture Notes in Mathematics. 1747, Springer-Verlag, Berlin, 2000.
	
	\bibitem{EO1} E.E. Enochs, O.M.G. Jenda. Gorenstein injective and projective modules. Mathematische Zeitschrift. 220, 611-633, 1995.
	
	\bibitem{EO} E.E. Enochs, O.M.G. Jenda. Balanced functors applied to modules.  Journal of Algebra. 92, 303-310, 1985.
	
	\bibitem{EO2} E.E. Enochs, O.M.G. Jenda. Relative Homological Algebra. De Gruyter Expositions in Mathematics. vol. 30, Walter de Gruyter GmbH \& Co. KG, Berlin, 2011.
	
	\bibitem{GZ}N. Gao, P, Zhang. Gorenstein derived categories. Journal of Algebra. 323, 2041-2057, 2010.
	
	\bibitem{He} Z. He. Gorenstein Objects in Extriangulated Categories. arXiv:2011.14552.

	\bibitem{H1} H. Holm. Gorenstein dervied functors. Proceedings of the American Mathematical Society. 132, 1913-1923, 2004. 
	
	\bibitem{H2} H. Holm. Gorenstein homological dimensions. Journal of Pure and Applied Algebra.  189, 167-193, 2004.
	
	\bibitem{JDP} J. Hu, D. Zhang, P. Zhou. Proper classes and Gorensteinness in extriangulated categories.  Journal of Algebra. 551, 23-60, 2020.
	
	\bibitem{JZP} J. Hu, D. Zhang, P. Zhou. Gorenstein homological dimensions for extriangulated categories.  arXiv:1908.00931.
	
	\bibitem{JDTP} J. Hu, D. Zhang, T. Zhao, P. Zhou. Balance of complete cohomology in extriangulated categories. arXiv:2004.13711v1.
	
	
	\bibitem{HY} H. Nakaoka, Y. Palu. Extriangulated categories, Hovey twin cotorsion pairs and model structures. Cahiers de Topologie et Differentielle Cateoriques. Volume LX-2, 117-193, 2019.
	
	\bibitem{WZ} W. Ren, Z. Liu. Gorenstein homological dimensions for triangulated categories. Journal of Algebra. 410, 258-276, 2014.
	
\end{thebibliography}
\end{document}